\documentclass[11pt]{amsart}

\usepackage{amscd}
\usepackage{amsmath, amssymb}
\usepackage{amsfonts}

\renewcommand{\leq}{\leqslant}
\renewcommand{\geq}{\geqslant}

\newcommand{\be}{\begin{equation}}
\newcommand{\ee}{\end{equation}}

\newcommand{\p}{\partial}
\newcommand{\ol}{\overline}
\newcommand{\ul}{\underline}

\begin{document}
\newtheorem{claim}{Claim}
\newtheorem{theorem}{Theorem}[section]
\newtheorem{lemma}[theorem]{Lemma}
\newtheorem{corollary}[theorem]{Corollary}
\newtheorem{proposition}[theorem]{Proposition}
\newtheorem{question}{question}[section]
\newtheorem{definition}[theorem]{Definition}
\newtheorem{remark}[theorem]{Remark}

\numberwithin{equation}{section}

\title[Prescribed curvature equations]
{The Dirichlet problem for a class of prescribed curvature equations}

\author[H. Jiao]{Heming Jiao}
\address{School of Mathematics and Institute for Advanced Study in Mathematics, Harbin Institute of Technology,
         Harbin, Heilongjiang, 150001, China}
\email{jiao@hit.edu.cn}

\author[Z. Sun]{Zaichen Sun}
\address{School of Mathematics, Harbin Institute of Technology,
         Harbin, Heilongjiang, 150001, China}
\email{16b912005@stu.hit.edu.cn}

\begin{abstract}
In this paper, we consider the Dirichlet problem for a class of prescribed curvature equations.
Both degenerate and non-degenerate cases are considered.
The existence of the $C^{1,1}$ regular graphic hypersurfaces
with prescribing a class of curvatures and constant boundary is proved for the degenerate case.

\end{abstract}

\maketitle

\section{Introduction}
One of classic problems in differential geometry is to find hypersurfaces with prescribed curvatures
and boundary data such as Plateau problem and its fully nonlinear generalizations (\cite{GS93,GS02}).
If the boundary is the graph of a given function $\varphi: \partial \Omega \rightarrow \mathbb{R}$,
where $\Omega$ is a domain in $\mathbb{R}^n$, such problems can be simplified to find a graphic hypersurfaces
in $\mathbb{R}^{n+1}$ which are equivalent to solving Dirichlet problems of the form
\begin{equation}
\label{js-1}
\left\{ \begin{aligned}
   f (\kappa[M_{u}]) & = \psi & \;\;\mbox{ in } &\Omega, \\
                 u &= \varphi & \;\;\mbox{ on }~ &\partial \Omega, &
\end{aligned} \right.
\end{equation}
where $M_u=\{(x,u(x)); x\in\Omega\}$ is the graphic hypersurface defined by the function $u$,
$\kappa[M_u] = (\kappa_1, \ldots, \kappa_n)$ are the principal curvatures of $M_u$ and $f$
is the curvature function. In particular, when $f = \kappa_1 + \cdots + \kappa_n$, $f = \sum_{i < j} \kappa_i \kappa_j$
and $f = \kappa_1 \cdots \kappa_n$, \eqref{js-1} are prescribed mean curvature, scalar curvature and
Gauss curvature equations respectively.

We call a $C^2$ regular hypersurface $M \subset \mathbb{R}^{n+1}$ $(\eta, n)$-convex if its principal curvature vector
$\kappa (X) \in \ol \Gamma$, where $\Gamma$ is a symmetric cone defined by
\[
\Gamma := \{\kappa = (\kappa_1, \ldots, \kappa_n) \in \mathbb{R}^n: \lambda_i = \sum_{j \neq i} \kappa_i > 0, i = 1, \ldots, n\}
\]
for all $X \in M$. A $C^2$ function $u: \Omega \rightarrow \mathbb{R}$ is called admissible if its graph
is $(\eta, n)$-convex.

Such hypersurface is useful to describe the boundaries of Riemannian manifolds
which have the homotopy type of a CW-complex (\cite{Sha86, Wu87}) and were studied by Sha \cite{Sha86,Sha87}, Wu \cite{Wu87}
and Harvey-Lawson \cite{HL13} intensively. It is worth mentioning that an $(\eta, n)$-convex hypersurface was called $(n-1)$-convex
in these references. Let $H (X)$ be the mean curvature of $M$ at $X \in M$. Define the $(0, 2)$-tensor field $\eta$ on $M$ by
\[
\eta = H g - h,
\]
where $g$ and $h$ are the first and second fundamental forms of $M$ respectively. Obviously a hypersurface $M$ is $(\eta, n)$-convex
if and only if $\eta$ is positive definite at each point of $M$. To measure the $(\eta, n)$-convexity, it is natural to introduce the
$(\eta, n)$-curvature at $X \in M$: $K_\eta (X) := \lambda_1 (X) \cdots \lambda_n (X)$. It is clear that
\[
K_\eta (X) = \det (g^{-1} \eta (X)).
\]

In this paper, we are concerned with the existence of graphic $(\eta, n)$-convex hypersurfaces with prescribed $(\eta, n)$-curvature and boundary.
In particular, we consider the Dirichlet problem
\begin{equation}
\label{js-2}
\left\{ \begin{aligned}
   K_\eta [M_u] & = \psi (X, \nu (X)) & \;\; X = (x, u(x)) , x \in \Omega, \\
                 u (x) &= 0 & \;\;~ x \in \partial \Omega,
\end{aligned} \right.
\end{equation}
where $\nu (X)$ denotes the upward unit normal vector to $M_{u}$ at $X \in M_{u}$.
In the current work, the function $\psi$ is allowed to vanish somewhere so that the equation
\eqref{js-2} is degenerate.
 
To obtain the second order boundary estimates, we usually need  
some geometric conditions on $\Omega$. A bounded domain $\Omega$ in $\mathbb{R}^n$ is called uniformly $(k-1)$-convex 
if there exists a positive constant
$K$ such that for each $x \in \partial \Omega$,
\[
(\kappa^b_1 (x), \ldots, \kappa^b_{n - 1} (x), K) \in \Gamma_{k},
\]
where $\kappa^b_1 (x), \ldots, \kappa^b_{n - 1} (x)$ are the principal curvatures of $\partial \Omega$ at
$x$ and $\Gamma_k$ is the G\r{a}rding's cone
\[
\Gamma_k=\{\kappa \in \mathbb R^n \ | \quad \sigma_m(\kappa)>0,
\quad  m=1,\cdots,k\}.
\] 
It is easy to see $\Gamma_2 \subset \Gamma$ and that $(n-1)$-convexity is equivalent to strict convexity.

Throughout the paper, $\Omega$ is always assumed to be uniformly $2$-convex for $n \geq 3$ and
and strictly convex for $n=2$.
Our main results are stated in the following theorem.
\begin{theorem}
\label{js-thm1}
Let $\Omega$ be a bounded in $\mathbb{R}^n$ with $\partial \Omega \in C^{2,1}$.
Suppose
\begin{equation}
\label{condition}
\psi^{1/(n-1)} = \psi^{1/(n-1)} (x, z, \nu) \in C^{1,1} (\mathbb{R}^n \times \mathbb{R} \times \mathbb{S}^n) \geq 0
\end{equation}
and $\psi_z \geq 0$.
Assume that that there exists an admissible sub-solution $\underline{u} \in C^{1,1} (\overline{\Omega})$ satisfying
$\kappa [M_{\underline{u}}] \in \Gamma$ and
\begin{equation}
\label{subsol}
\left\{ \begin{aligned}
   K_\eta [M_{\underline{u}}] & \geq \psi (\ul X, \ul \nu (\ul X))  & \;\; \ul X = (x, \ul u(x)), x \in \Omega,\\ 
                 \ul u (x) &= 0  & \;\;~ x \in \partial \Omega,
\end{aligned} \right.
\end{equation}
where $\nu (\ul X)$ denotes the upward unit normal vector to $M_{\ul u}$ at $\ul X \in M_{\ul u}$.
In addition, there exists a function $\ul \psi = \ul \psi (x, z) \in C^0 (\ol \Omega \times \mathbb{R}) \geq 0$
satisfying $\ul \psi \not\equiv 0$ on $\ol \Omega \times [-\epsilon, 0]$ for any $\epsilon > 0$ such that
\begin{equation}
\label{add-8}
\psi (x, z, \nu) \geq \ul \psi (x, z) \mbox{ for }(x, z, \nu) \in \ol \Omega \times \mathbb{R}\times \mathbb{S}^n.
\end{equation}
Then there exists a unique admissible solution $u \in C^{1,1} (\ol \Omega)$ of \eqref{js-2}.
\end{theorem}
To prove Theorem \ref{js-thm1}, we need the solvability of the non-degenerate equation \eqref{js-2} which is
stated in the next theorem.
\begin{theorem}
\label{js-thm2}
Let $\Omega$ be a bounded domain in $\mathbb{R}^n$ with smooth boundary $\partial \Omega$.
Suppose $\psi = \psi (x, z, \nu) \in C^\infty (\ol \Omega \times \mathbb{R} \times \mathbb{S}^n) > 0$ and $\psi_z \geq 0$.
Assume that there exists a subsolution $\ul u$ as in Theorem \ref{js-thm1}. Then there exists a unique admissible
solution $u \in C^{\infty} (\ol \Omega)$ of \eqref{js-2}.
\end{theorem}
The non-degenerate prescribed curvature equations have received extensively studies.
In particular, Caffarelli-Nirenberg-Spruck \cite{CNSV} studied the non-degenerate prescribed curvature equations
of general form with $\psi$ independent of $u$ and $\nu$ and $\varphi \equiv \mbox{ constant }$ in strictly convex domains.
Ivochkina \cite{Ivochkina90, Ivochkina91} studied the
Dirichlet problem for the prescribed curvature equation
\begin{equation}
\label{sigmak}
\sigma_k (\kappa) = \psi (x, u),
\end{equation}
where $\sigma_{k}$ is the $k$-th elementary symmetric function
\[
\sigma_{k} (\kappa) = \sum_ {i_{1} < \cdots < i_{k}}
\kappa_{i_{1}} \cdots \kappa_{i_{k}}, \quad \text{for $k=1,2,\cdots,n$},
\]
and her results were generalized by Lin-Trudinger \cite{LT94b} and Ivochkina-Lin-Trudinger \cite{ILT96} to the equation
\begin{equation}
\label{sigmak}
(\sigma_k/\sigma_l) (\kappa) = \psi (x, u).
\end{equation}
The reader is referred to \cite{CNSIV, CJ20, GG02, GLL12, GLM09, GRW15, RW19, RW20, SX17} for
more researches about non-degenerate curvature equations.

The degenerate curvature equations arise naturally in many geometric problems such as the degenerate Weyl problem studied by Guan-Li \cite{GL1} and Hong-Zuily \cite{HZ},
where the equations can be viewed as some two dimensional degenerate Monge-Amp\`ere type curvature equation. Guan-Li \cite{GL2} considered the prescribed
Gauss curvature measure problem, which also can be rewritten as some degenerate Monge-Amp\`ere type curvature equation.
As we know, the best regularity one can expect for degenerate curvature
equations is $C^{1,1}$. Recently, Jiao-Wang \cite{JW21} proved the existence of solutions in $C^{1,1} (\ol \Omega)$
for degenerate prescribed $k$-curvature equations \eqref{sigmak}. Guan-Zhang \cite{GZ} considered another class of curvature type equations which is the combination of  $\sigma_k$.

When $n=2$, equation \eqref{js-2} is the classic prescribed Gauss curvature equation. For general $n$,
\begin{equation}
\label{js-4}
K_\eta [M] = K_\eta (\kappa) = \sum_{i=2}^n \sigma_1 (\kappa)^{n-i} \sigma_i (\kappa).
\end{equation}
In \cite{CJ20} Chu-Jiao considered the curvature estimates for star-shaped hypersurfaces $M$, i.e., $M$ can be
represented by a radial graph of positive function on $\mathbb{S}^n$, satisfying the equation
\begin{equation}
\label{cj-1}
\sigma_k (\lambda (\eta(X)) = f (X, \nu (X)), \quad \text{for $X \in M$},
\end{equation}
where $\lambda (\eta)$ denote the eigenvalues of $\eta$ with respect to the metric $g$. When $k=n$, $\sigma_{k} (\lambda(\eta))$
is the $(\eta, n)$-curvature of $M$. It is an interesting problem to consider the Dirichlet problem \eqref{js-2} with
$K_\eta [M_u]$ replaced by $\sigma_{k} (\lambda(\eta))$.
Recently Yuan \cite{Yuan20} proved an inequality
for concave functions which may be applied to derive the second order interior estimates for more general equations
whose elliptic cones are not the positive cone $\Gamma_n$.

The corresponding Hessian type equation of \eqref{js-2},
\begin{equation}
\label{n-1MA}
\det (\Delta u I - D^2 u) = \psi (x, u, Du)
\end{equation}
is called $(n-1)$ Monge-Amp\`ere equation. Recently Dinew \cite{Dinew20}, Chu-Jiao \cite{CJ20}
studied the interior estimates for some Hessian type equations including \eqref{n-1MA} and its Dirichlet problem
was studied later \cite{JL20}. The complex analogue of \eqref{n-1MA} has been studied extensively since it is related to the Gauduchon conjecture
(see \cite[\S IV.5]{Gauduchon84}) which was solved by Sz\'ekelyhidi-Tosatti-Weinkove \cite{STW17} in complex geometry.
For more references, the reader is referred to \cite{FWW10,FWW15,Popovici15,Szekelyhidi18,TW17,TW19} and references therein.

Now we make some remarks on the conditions of Theorem \ref{js-thm1}. In general, the Dirichlet problem is not always solvable
without the existence of a subsolution. In this paper, we only use the subsolution to derive the lower bound of $u$ and the gradient
estimates on the boundary. Compared with Hessian type equations, the subsolution is not as powerful to construct barriers because
the principal curvatures are eigenvalues of a much more complicated matrix. We make use of the $2$-convexity ($n\geq 3$) of $\Omega$ to
construct suitable barriers for second order estimates on the boundary. However, more natural condition on the domain $\Omega$
should be that there exists a positive constant $K$ such that
\[
(\kappa^b_1 (x), \ldots, \kappa^b_{n-1} (x), K) \in \Gamma
\]
for any $x \in \partial \Omega$, where $\kappa^b_1 (x), \ldots, \kappa^b_{n-1} (x)$ are the principal curvatures of $\partial \Omega$
at $x$. Such domain is called admissible. It is of interest to ask if we can establish the second boundary estimates
in admissible domains. For non-negative functions, the condition
\begin{equation}
\label{add-1}
\psi^{1/(n-1)} \in C^{1,1} (\ol \Omega \times \mathbb{R} \times \mathbb{S}^n)
\end{equation}
is a little weaker than that $\psi^{1/n} \in C^{1,1} (\ol \Omega \times \mathbb{R} \times \mathbb{S}^n)$.
Guan-Trudinger-Wang \cite{GTW99} proved the $C^{1,1}$ regularity for the solution of degenerate Monge-Amp\`{e}re equation
\[
\det (D^2 u) = \psi (x) \geq 0  \mbox{ in } \Omega
\]
with $u = \varphi$ on $\partial \Omega$ under the condition $\psi^{1/(n-1)} \in C^{1,1} (\ol \Omega)$.
Examples of Wang \cite{Wang95} show that the condition is optimal. Note that the assumption \eqref{condition} in Theorem \ref{js-thm1}
is stronger than \eqref{add-1}. Let $\tilde{\psi} := \psi^{1/(n-1)}$. Indeed, \eqref{condition} implies \eqref{add-1} and that
\begin{equation}
\label{add-2}
|D \tilde{\psi} (x, z, \nu)| \leq C_{\mu_0} \sqrt{\tilde{\psi} (x, z, \nu)}
\end{equation}
for any $(x, z, \nu) \in \ol \Omega \times [-\mu_0, \mu_0] \times \mathbb{S}^n$ and some positive constant $C_{\mu_0}$ depending
only on $\|\tilde{\psi}\|_{C^{1,1} (\ol \Omega \times [-\mu_0, \mu_0] \times \mathbb{S}^n)}$
(seeing Lemma 3.1 in \cite{Blocki03}) which is only used to derive the gradient estimates. It is easy to see \eqref{add-2} yields that
\begin{equation}
\label{add-3}
\psi^{1/n} \in C^{1} (\ol \Omega \times \mathbb{R} \times \mathbb{S}^n).
\end{equation}
The existence of the function $\ul \psi$ satisfying the assumptions in Theorem \ref{js-thm1} is only used to establish the
estimates for normal-normal second order derivatives. If $\psi$ does not depend on $\nu$, we do not need the existence of $\ul \psi$
in Theorem \ref{js-thm1}.

The key step to prove Theorem \ref{js-thm1} is the establishment of the \emph{a priori} $C^2$ estimates for non-degenerate
equation \eqref{js-2} independent of the lower bound of $\psi$.
Thus, the existence of solutions in $C^{1, 1} (\ol \Omega)$ can be proved by a non-degenerate approximation.
The main parts of this work are the global and boundary estimates for second order derivatives.
Since the function $\psi$ may depend on the unit normal $\nu$, there are more troublesome terms when
we differentiating the equation \eqref{js-2}. We shall apply an idea of \cite{CJ20} to overcome this by using
some properties of $\sigma_{n-1}$ in the global estimates. Another difficulty in global estimates is that we cannot use the
concavity of $\sigma_n^{1/n}$ directly since we only assume \eqref{add-1}. For this we use an idea of \cite{JW21}
to apply a lemma from the concavity of $(\sigma_k/\sigma_1)^{1/(k-1)}$ which was proved by Guan-Li-Li \cite{GLL12}.
The main challenge for the boundary estimates
is from the degeneracy of the equation.
The key is a calculation of the linearized operator acting on the tangential derivatives
of the solution, namely, Lemma \ref{BC2-lem1}. We remark that the special
structure of the $(\eta, n)$-curvature plays an important role in the proof of Lemma \ref{BC2-lem1}.

It is worth mentioning that, in the non-degenerate case,
if $\psi \equiv \psi (x, u) > 0$ depends only on $X = (x, u(x)) \in M_u$,
the second order
boundary estimates can be obtained by an easy generalization of \cite{CNSV}. However, when
$\psi$ also depends on $\nu$, some uncontrollable terms arise when the linearized operator
acts on their barrier (seeing (4.6) in \cite{CNSV}). We do not have any idea to generalize
the method of \cite{CNSV} in such case yet.

This paper is organized as follows. In Section 2
some preliminaries are provided. The $C^1$ estimates are established in Section 3.
Section 4 and 5 are devoted to the global and boundary estimates for second order
derivatives respectively.

\bigskip

{\bf Acknowledgement.}
The first author is supported by the NSFC (Grant No. 11871243)
and the Natural Science Foundation of Heilongjiang Province (Grant No. LH2020A002).

\section{Preliminaries}

Throughout this paper, $\phi_i = \frac{\p \phi}{\p x_i}$, $\phi_{ij} = \frac{\p^2 \phi}{\p x_i \p x_j}$,
$D\phi=(\phi_{1}, \cdots, \phi_{n})$ and $D^2 \phi = (\phi_{ij})$ denote the ordinary first and second order derivatives, gradient and Hessian of
a function $\phi \in C^2 (\Omega)$ respectively.

A graphic hypersurface $M_u$ in $\mathbb{R}^{n+1}$ is a codimension one submanifold which can be written as a graph
\[ M_u=\{X=(x, u(x))| x\in\mathbb{R}^n\}.\]
Let
$\epsilon_{n+1} = (0, \cdots, 0, 1) \in \mathbb{R}^{n+1}$, then the height function of $M_u$ is $u(x)=\langle X, \epsilon_{n+1}\rangle$.
It's easy to see that the induced metric and second fundamental form of $M$ are given
by
$$g_{ij}=\delta_{ij}+ u_i u_j, \ \  1\leq i,j\leq n,$$
and
\[h_{ij}=\frac{u_{ij}}{\sqrt{1+|Du|^2}},\]
while the upward unit normal vector field to $M$ is
\[\nu=\frac{(-Du, 1)}{\sqrt{1+|Du|^2}}.\]
By straightforward calculations, we have the principle curvatures of $M_u$ are eigenvalues of the matrix
\[
\frac{1}{w} \big(I - \frac{Du \otimes Du}{w^2}\big) D^2 u
\]
or the symmetric matrix
$A [u] = (a_{ij})$:
\begin{equation}
\label{matrix}
a_{ij}=\frac{1}{w}\gamma^{ik}u_{kl}\gamma^{lj},
\end{equation}
where $\gamma^{ik}=\delta_{ik}-\frac{u_iu_k}{w(1+w)}$ and $w=\sqrt{1+|Du|^2}.$ Note that $(\gamma^{ij})$ is invertible with inverse
$\gamma_{ij}=\delta_{ij}+\frac{u_iu_j}{1+w},$ which is the square root of $(g_{ij})$.

For $r \in S^{n \times n}$ and $p \in \mathbb{R}^n$, define
\[
\lambda (r, p) = \lambda \Big(\big(I - \frac{p \otimes p}{1 + |p|^2}\big) r\Big)
\]
and
\[
S_k (r, p) = \sigma_k (\lambda (r,p)).
\]
As in \cite{ILT96}, we introduce the notations as follows. For $p \in \mathbb{R}^n$, $i = 1, \ldots, n$, let
$p(i)$ be the vector obtained by setting $p_i = 0$, $r (i)$ the matrix obtained by setting the $i^th$ row
and column to zero and $r (i,i)$ the matrix obtained by setting $r_{ii} = 0$. Denote
\[
S_{k,i} (r,p) = S_k (r(i), p(i)).
\]
Ivochkina-Lin-Trudinger \cite{ILT96} proved the formula
\begin{equation}
\label{ILT}
S_k (r, p) = \frac{1 + |p(i)|^2}{1+|p|^2} r_{ii} S_{k-1; i} (r, p) + O (|r (i,i)|^k)
\end{equation}
for all $1\leq i \leq n$ and $1\leq k \leq n$ and $S_0$ is defined by $S_0 \equiv 1$.
Combining \eqref{js-4} and \eqref{ILT}, we have
\begin{equation}
\label{js-5}
\begin{aligned}
K_\eta (M_u) = \,& \frac{1}{w^n}\sum_{i=2}^n S_1^{n-i} (D^2 u, Du) S_i (D^2 u, Du)\\
  = \,& \frac{1}{w^n} \left(\frac{1 + |Du(n)|^2}{1 + |Du|^2}\right)^{n-1}  S_{1;n} (D^2 u, Du) u_{nn}^{n-1}\\
   & + \sum_{i=1}^{n-2} P_i u_{nn}^{i} + P_0,
\end{aligned}
\end{equation}
where $P_i$ depend only on $u_{\alpha \beta}$ ($\alpha + \beta < 2n$) and $Du$.

For $\kappa \in \mathbb{R}^n$ let
\begin{equation}
\label{lambda}
\lambda_i = \sum_{j \neq i} \kappa_j, i = 1, \ldots, n.
\end{equation}
Define the function $f (\kappa)$ on the cone $\Gamma$
by
\begin{equation}
\label{def-h}
f (\kappa) := \lambda_1 \cdots \lambda_n.
\end{equation}
The function $f$ satisfies the following properties which may be used in the following sections. Their proofs can be found in \cite{JL20}.
\begin{equation}
\label{cj-10}
f_i (\kappa) = \frac{\partial f (\kappa)}{\partial \kappa_i} > 0, \mbox{ in } \Gamma, i = 1, \ldots, n;
\end{equation}
\begin{equation}
\label{cj-9}
f^{1/n} (\kappa) \mbox{ is concave in } \Gamma;
\end{equation}
\begin{equation}
\label{js-3}
f > 0 \mbox{ in } \Gamma \mbox{ and } f = 0 \mbox{ on } \partial \Gamma;
\end{equation}
\begin{equation}
\label{cj-6}
f_j (\kappa) \geq \delta_0 \sum_i f_i (\kappa), \mbox{ if } \kappa_j < 0, \forall \kappa \in \Gamma
\end{equation}
for some positive constant $\delta_0$ depending only on $n$ and for any constant $A > 0$ and any compact set $K$ in $\Gamma$
there is a number $R = R (A, K)$ such that
\begin{equation}
\label{cj-2}
f (\kappa_1, \ldots, \kappa_{n-1}, \kappa_n + R) \geq A, \mbox{ for all } \kappa \in K.
\end{equation}
We also need some some algebraic equalities and inequalities of $\sigma_k$ (seeing \cite{LT94}):
\[
\sum_i \sigma_{k-1; i}(\kappa) = (n-k+1) \sigma_{k-1}(\kappa),
\]
\[
\sum_i \sigma_{k-1; i} (\kappa) \kappa_i = k \sigma_k (\kappa),
\]
and
\begin{equation}
\label{NM-1}
\sigma_{k - 1} (\kappa) \geq c_0 \sigma_k^{1 - 1/(k-1)} (\kappa) \sigma_1^{1/(k-1)} (\kappa)
\end{equation}
\begin{equation}
\label{NM-2}
\sigma_{1} (\kappa) \geq c_0 \sigma_k^{1/k} (\kappa)
\end{equation}
for any $\kappa \in \Gamma_k$ and
some positive constant $c_0$ depending only on $n$ and $k$. Note that the last two inequalities are consequences of the Newton-Maclaurin inequalities.

Now, let $\{e_1,e_2,\cdots,e_n\}$ be a local orthonormal frame on $TM_u$. We will use $\nabla$ to denote
the induced Levi-Civita connection on $M.$ For a function $v$ on $M_u$, we denote $\nabla_i v=\nabla_{e_i}v,$
$\nabla_{ij} v = \nabla^2 v (e_i, e_j),$ etc in this paper.
Thus, we have
\[|\nabla u|=\sqrt{g^{ij}u_{i}u_{j}}=\frac{|Du|}{\sqrt{1+|Du|^2}}.\]

Using normal coordinates, we also need the following well known fundamental equations for a hypersurface $M$ in $\mathbb{R}^{n+1}:$
\begin{equation}\label{Gauss}
\begin{aligned}
\nabla_{ij} X = \,& h_{ij}\nu \quad {\rm (Gauss\ formula)}\\
\nabla_i \nu= \,& -h_{ij} e_j \quad {\rm (Weigarten\ formula)}\\
\nabla_k h_{ij} = \,& \nabla_j h_{ik} \quad {\rm (Codazzi\ equation)}\\
R_{ijst} = \,& h_{is}h_{jt}-h_{it}h_{js}\quad {\rm (Gauss\ equation)},
\end{aligned}
\end{equation}
where $h_{ij} = \langle \nabla_{e_i} e_j, \nu\rangle$,
$R_{ijst}$ is the $(4,0)$-Riemannian curvature tensor of $M$, and the derivative here is covariant derivative with respect to the metric on $M$.
Therefore, the Ricci identity becomes,
\begin{equation}\label{Commutation formula}
\nabla_{ij} h_{st} = \nabla_{st} h_{ij}+(h_{mt}h_{sj}-h_{mj}h_{st})h_{mi}+(h_{mt}h_{ij}-h_{mj}h_{it})h_{ms}.
\end{equation}

For $p \in \mathbb{R}^n$, define
\[
\Gamma_k (p) := \left\{r \in S^{n \times n}: \lambda (r, p) := \lambda \left(\left(I - \frac{p \otimes p}{1 + |p|^2}\right) r\right) \in \Gamma_k\right\}.
\]
To construct the barrier for second order estimates on the boundary, we need the following proposition proved by Ivochkina \cite{Ivochkina91}
first. Its proof can be also found in \cite{JW21}.
\begin{proposition}
\label{prop-I}
Suppose $k \leq n-1$.
For any $p\in\mathbb{R}^n$, we have $\Gamma_{k+1} (0) \subset \Gamma_k (p)$. Suppose  $r \in \Gamma_{k+1}$ is a $n\times n$ matrix,
then we have
\begin{equation}
\label{BC2-18}
\sigma_j (\lambda (r, p)) \geq \frac{1}{1 + |p|^2} \sigma_j (\lambda (r))
\end{equation}
for  $j = 1, \ldots, k$.
\end{proposition}

\section{$C^1$ estimates}

In this and the following sections, we assume $\psi > 0$ and establish the $C^2$ estimates independent
of $\inf \psi$.

In this section, we consider the $C^1$ estimates for the admissible solution of \eqref{js-2}.
Since $M_u$ is $(\eta, n)$-convex, we find the mean curvature of $M_u$, $H [M_u] > 0$.
It follows from the maximum principle that $u \leq 0$ in $\Omega$ since $u = 0$ on $\partial \Omega$.
By the maximum principle again we get $u \geq \ul u$ in $\Omega$. It follows that
\begin{equation}
\label{js-7}
\sup_{\ol \Omega} |u| + \sup_{\partial \Omega} |D u| \leq C,
\end{equation}
for some positive constant $C$ depending only on $\|\ul u\|_{C^{0,1} (\ol \Omega)}$.
Next we establish the global gradient estimates to prove
\begin{theorem}
\label{gradient}
Let $u \in C^3 (\Omega) \cap C^1 (\ol \Omega)$ be an admissible solution
of \eqref{js-2}. Suppose \eqref{add-3} and $\psi_z \geq 0$.
Then there exists a positive constant $C$ depending only on $n$, $\|u\|_{C^0 (\ol \Omega)}$ and
$\|\psi^{1/n}\|_{C^{1} (\overline{\Omega} \times [- \mu_0, \mu_0]\times \mathbb{S}^n)}$ but independent of $\inf \psi$ such that
\begin{equation}
\label{gradient-1}
\sup_{\ol \Omega} |Du| \leq C (1 + \sup_{\partial \Omega} |Du|),
\end{equation}
where $\mu_0 := \|u\|_{C^0 (\overline{\Omega})}$.
\end{theorem}
\begin{proof}
Suppose $\epsilon_1, \ldots, \epsilon_{n+1}$ is a standard basis of $\mathbb{R}^{n+1}$. Let
\[
Q=\log w+B\langle X,\epsilon_{n+1}\rangle = \log w+Bu,
\]
where $w=\frac{1}{\langle\nu,\epsilon_{n+1}\rangle} = \sqrt{1 + |Du|^2}$ and $B$ is a positive constant
sufficiently large to be determined later.
Suppose the maximum value of $Q$ is achieved at an interior point $x_0 \in \Omega$.
We rotate $\epsilon_1, \ldots, \epsilon_n$ to satisfy, at $x_0$, $u_1 = |Du|$, $u_j = 0$ for $j \geq 2$.
Define $e_i = \gamma^{is} \tilde{\partial}_s$, where $\tilde{\partial}_s := \epsilon_s + u_s \epsilon_{n+1}$
for $1 \leq s \leq n$, $i = 1, \ldots, n$. It is clear that $\{e_1,e_2,\cdots,e_n\}$ is an orthonormal frame on $M_u$ and
at $X_0 = (x_0, u (x_0))$, satisfy
$$\nabla_1u=\frac{|Du|}{w}=|\nabla u|, \nabla_i u=u_i=0, \text{ for } i\geq 2.$$
We may further rotate $\epsilon_2, \ldots, \epsilon_n$ such that $\{u_{ij}\}_{i,j \geq 2}$ is
diagonal at $x_0$.
We have
$$\nabla_i w=\frac{1}{(\langle\nu,\epsilon_{n+1}\rangle)^2}h_{im}\nabla_m u.$$
Thus, at the maximum point $x_0$ of $Q$, we get
\begin{equation}
\label{4.2}
\nabla_i Q = wh_{i1}\nabla_1 u + B \nabla_i u=0.
\end{equation}
Taking $i=1$ and $i\geq 2$ respectively, we get, at $X_0$, $$wh_{11}=-B, h_{i1}=0.$$
Since at $X_0$,
\[
h_{11} = \frac{u_{11}}{w^3}, h_{1i} = \frac{u_{1i}}{w^2} \mbox{ and } h_{ij} = \frac{u_{ij}}{w}
  \mbox{ for } i, j \geq 2,
\]
the matrix $\{h_{ij}\}$ is diagonal at $X_0$.
At $X_0$, we have
\begin{equation}
\label{js-5}
\nabla_{ii} Q = w^2(h_{im} \nabla_m u)^2+w \nabla_i h_{i1} \nabla_1 u+wh_{im} \nabla_{mi} u + B\nabla_{ii} u \leq 0.
\end{equation}
Let $\{y_k\}_{k=1}^n$ denote the standard local coordinate system on $\mathbb{S}^n$ near $(0, \ldots, 0, 1)$.
Let
\[
F^{ij} = \frac{\partial f (\lambda (h))}{h_{ij}}.
\]
By Codazzi equations, Weingarten formula and differentiating the equation \eqref{js-2}, we have, at $X_0$,
\begin{equation}
\label{js-6}
\begin{aligned}
F^{ii} \nabla_i h_{i1} = \,& F^{ii} \nabla_1 h_{ii} = e_1 (\psi)\\
  = \,& \psi_{x_j} \nabla_1 x_j + \psi_u \nabla_1 u + \partial_{y_k} \psi \nabla_1 \nu_k\\
  = \,& \frac{\psi_{x_1}}{w} + \psi_u \nabla_1 u - \frac{\partial_{y_1} \psi h_{11}}{w}\\
  = \,& \frac{\psi_{x_1}}{w} + \psi_u \nabla_1 u + \frac{B \partial_{y_1} \psi}{w^2}.
\end{aligned}
\end{equation}
Using \eqref{add-3}, \eqref{js-5}, \eqref{js-6}, $\psi_u \geq 0$ and that $\nabla_{ij} u = \frac{h_{ij}}{w}$, we obtain, at $X_0$,
\begin{equation}\label{4.3}
\begin{aligned}
0 \geq \,& w^2 F^{11}h_{11}^2 (\nabla_1 u)^2 + w \nabla_1 u (e_1 \psi) + F^{ii}h_{ii}^2+ \frac{nB \psi}{w}\\
      \geq \,& F^{11}(\nabla_1 u)^2 B^2 + \psi_{x_1} \nabla_1 u + \frac{B \nabla_1 u \partial_{y_1} \psi}{w}\\
  \geq \,& B^2 F^{11}(\nabla_1 u)^2 - C \Big(1 + \frac{B}{w}\Big) \psi^{1-1/n}.
\end{aligned}
\end{equation}
By \eqref{cj-6} and that $\kappa_1$ is negative, we get
\[
F^{11}\geq \delta_0 \sum F^{ii} = \delta_0 (n-1) \sigma_{n-1} (\eta) \geq \delta_1 \psi^{1-1/n}
\]
for some positive constant $\delta_1$ depending only on $n$.
Therefore, from \eqref{4.3}, we have
\[
B^2 \delta_1 |\nabla u|^2 - C \Big(1 + \frac{B}{w}\Big) \leq 0.
\]
Since we can assume $w$ is sufficiently large, using $\frac{1}{w^2}+|\nabla u|^2=1$, we can assume $|\nabla u|^2\geq 1/2$. We obtain the desired estimate
by fixing $B$ sufficiently large.
\end{proof}
Combining \eqref{js-7} and \eqref{gradient-1}, we obtain the \emph{a priori} $C^1$ estimates.

\section{Global estimates for second order derivatives}
In this section, we deal with the global estimates for second order derivatives.
Since when $n=2$, \eqref{js-2} becomes classic two-dimensional Monge-Amp\`{e}re equation or a 2-Hessian equation and
the second order estimates follows by Proposition 2.3 in \cite{GRW15}, it suffices to consider the case when $n \geq 3$.
Therefore, in the following of this section, we assume $n \geq 3$.

\begin{theorem}
\label{Thm-second-interior}
Suppose \eqref{add-1} holds.
Let $u \in C^4 (\Omega) \cap C^2 (\overline{\Omega})$ be an admissible solution of \eqref{js-2}.
Then there exists a positive constant $C$ depending on $n$, $\|u\|_{C^1 (\overline{\Omega})}$
and $\|\psi^{1/(n-1)}\|_{C^{2} (\overline{\Omega} \times [- \mu_0, \mu_0]\times \mathbb{S}^n)}$ satisfying
\begin{equation}
\label{S-1}
\sup_{\ol \Omega} |D^2 u| \leq C \Big(1 + \sup_{\partial \Omega} |D^2 u|\Big),
\end{equation}
where $\mu_0 := \|u\|_{C^0 (\overline{\Omega})}$.
\end{theorem}
By Lemma 3.2 of \cite{GLL12}, we have
the following lemma.
\begin{lemma}
\label{lem-S-2}
Let $\alpha = \frac{1}{n-1}$. Then we have
\begin{equation}
\label{GLL-1}
\sum_s \sigma_n^{ij, pq} \nabla_s \eta_{ij} \nabla_s \eta_{pq} \leq (1 - \alpha) \frac{|\nabla \psi|^2}{\psi}
  + 2 \alpha \frac{\langle\nabla \psi, \nabla H\rangle}{H} - (1 + \alpha) \frac{\psi |\nabla H|^2}{H^2},
\end{equation}
where
\[
\sigma^{ij, pq}_n = \frac{\partial^2 \sigma_n (\lambda(\eta))}{\partial \eta_{ij} \partial \eta_{pq}},
\]
$\nabla \psi$ denotes the covariant gradient of $\psi$ when $\psi = \psi (X, \nu (X))$ is only regarded as
the function of $X \in M_u$ and $H = H (X)$ is the mean curvature of $M_u$ at $X \in M_u$.
\end{lemma}
\begin{proof}[Proof of Theorem \ref{Thm-second-interior}]
Let $v := 1/w = \langle\nu,\epsilon_{n+1}\rangle$.
There exists a positive constant $a$ depending only on $\|Du\|_{C^0 (\overline{\Omega})}$ such that $v \geq 2a$.
As in \cite{JW21}, we consider the test function
\[
V (X, \xi) := (v - a)^{-1} \exp\left\{\frac{\delta}{2} |X|^2\right\} H,
\]
where $X \in M_u$, $\xi \in T_X M_u$ is a unit vector and $\delta$ is a positive constant to be determined later.
Suppose that the maximum value of $V$ is achieved at a point $X_0 = (x_0, u(x_0)) \in M$, $x_0 \in \Omega$ and $\xi_0 \in T_{X_0} M_u$.
We choose a local orthonormal frame $\{e_{1},e_{2},\cdots,e_{n}\}$ near $X_{0}$ such that $\xi = e_1 (X_0)$,
\[
h_{ij} = \delta_{ij}h_{ii} \ \ \text{and} \ \
h_{11} \geq h_{22} \geq \cdots \geq h_{nn} \ \   \text{at $X_{0}$}.
\]
Therefore, at $X_0$, taking the covariant derivatives twice with respect to
\[
\log H + \log (v-a)^{-1} + \frac{\delta}{2} |X|^2,
\]
we have
\begin{equation}
\label{S-3}
 - \frac{\nabla_i v}{v-a} + \delta \langle X, e_i \rangle + \frac{\nabla_i H}{H} = 0,
  i = 1, \ldots, n
\end{equation}
and
\begin{equation}
\label{S-4}
0 \geq F^{ii} \Big\{
   - \frac{\nabla_{ii} v}{v- a} + \frac{(\nabla_i v)^2}{(v - a)^2}
   + \delta \big(1 + \kappa_i\langle X, \nu\rangle\big) + \frac{\nabla_{ii} H}{H} - \frac{(\nabla_i H)^2}{H^2}\Big\},
\end{equation}
where $\kappa_1, \ldots, \kappa_n$ are principal curvatures of $M_u$ at $X_0$ and
\[F^{ij}=\frac{\partial f (\lambda(h))}{\partial h_{ij}}.\]
Let
\[
\hat{F}^{ij} := \frac{\partial \sigma_n (\lambda (\eta))}{\partial \eta_{ij}} \mbox{ and } f_i = \frac{\partial f (\kappa)}{\partial \kappa_i}.
\]
It is easy to see
\[
F^{ij} = \sum_{s=1}^n \hat{F}^{ss} \delta_{ij} - \hat{F}^{ij} \mbox{ at } X_0.
\]
By the Weingarten equation, we have
\begin{equation}
\label{S-12}
\nabla_i v = - h_{im} \langle e_m, \epsilon_{n+1}\rangle
   = - h_{im} \nabla_m u \text{ and }
F^{ii} (\nabla_i v)^2 \leq C f_i \kappa_i^2.
\end{equation}
Next, by Gauss formula, the equation \eqref{js-2} and \eqref{add-1}, we have
\begin{equation}
\label{S-8}
\begin{aligned}
F^{ii} \nabla_{ii} v = \,& - F^{ii} \nabla_m h_{ii} \nabla_m u - f_i \kappa_{i}^2 v\\
  = \,& - v f_i \kappa_i^2 - \langle\nabla \psi, \nabla u\rangle
  \leq - v f_i \kappa_i^2 + C H \psi^{1-1/(n-1)}.
\end{aligned}
\end{equation}
where the positive constant $C$ depends on $\|\psi^{1/(n-1)}\|_{C^1}$ and $\|u\|_{C^1(\overline{\Omega})}$.
By \eqref{Commutation formula} and differentiating the equation \eqref{js-2} twice, we find
\begin{equation}
\label{S-9}
\begin{aligned}
F^{ii} \nabla_{ii} H = \,&
 \sum_{s=1}^n F^{ii} \nabla_{ss} h_{ii} - H f_i \kappa_i^2 + n \psi \sum_{s=1}^n h_{ss}^2 \\
  \geq \,& \sum_{s=1}^n\nabla_{ss} \psi - \sum_{s=1}^n\sigma_n^{ij, pq} \nabla_s \eta_{ij} \nabla_s \eta_{pq}
     - H f_i \kappa_i^2 + \psi H^2.
\end{aligned}
\end{equation}
Let $\alpha := \frac{1}{n-1}$.
Using \eqref{add-1} and Codazzi equation, we have
\begin{equation}
\label{add-4}
\begin{aligned}
\sum_{s=1}^n\nabla_{ss} \psi \geq \,& (1-\alpha)\frac{|\nabla \psi|^2}{\psi} - C (|\nabla X|^2
  + |\nabla \nu|^2) \psi^{1-\alpha}\\
  & + H \psi_{X_k} \nu_k - \sum_{s=1}^n \nabla_s h_{sj} (d_\nu \psi)(e_j)\\
  \geq \,& (1-\alpha)\frac{|\nabla \psi|^2}{\psi} - C H^2 \psi^{1-\alpha} - \nabla_j H (d_\nu \psi)(e_j).
\end{aligned}
\end{equation}
By \eqref{add-1}, \eqref{S-3} and \eqref{S-12}, we get
\begin{equation}
\label{S-10}
\begin{aligned}
\frac{\nabla_j H (d_\nu \psi)(e_j)}{H} = \left(\frac{\nabla_i v}{v-a} - \delta \langle X, e_i \rangle\right) (d_\nu \psi)(e_j)
   \leq C (H + \delta) \psi^{1-\alpha}.
\end{aligned}
\end{equation}
Next, by \eqref{S-3} and Cauchy-Schwarz inequality, we have, for any $\epsilon > 0$,
\begin{equation}
\label{S-13}
\frac{1}{H^2} F^{ii} (\nabla_i H)^2
   \leq (1+\epsilon) \frac{F^{ii} (\nabla_i v)^2}{(v - a)^2}
     + C (1 + 1/\epsilon) \delta^2 \sum F^{ii}.
\end{equation}
Combining \eqref{GLL-1}-\eqref{S-13}, we obtain
\begin{equation}
\label{S-14}
\begin{aligned}
0 \geq \,& \frac{a}{v-a} f_i \kappa_i^2 - \epsilon \frac{F^{ii} (\nabla_i v)^2}{(v - a)^2}
       - C H \psi^{1-\alpha} - 2 \alpha \frac{\langle \nabla \psi, \nabla H \rangle}{H^2}\\
  & + \left(\delta - C (1 + 1/\epsilon) \delta^2\right) \sum F^{ii}\\
  \geq \,& \left(\frac{a}{v-a} - \frac{C \epsilon}{(v-a)^2}\right) f_i \kappa_i^2
    + \left(\delta - C (1 + 1/\epsilon) \delta^2\right) \sum F^{ii} - C H \psi^{1-\alpha}
\end{aligned}
\end{equation}
provided $H$ is sufficiently large. Taking sufficiently small $\epsilon$ to satisfy $C \epsilon < a (v - a)/2$ and
$\delta$ sufficiently small such that $ C (1 + 1/\epsilon) \delta^2 < \delta/2$ in \eqref{S-14},
we get
\begin{equation}
\label{S-2}
\begin{aligned}
0 \geq \,& \frac{\delta}{2} \sum F^{ii} + \frac{a}{2 (v-a)} f_i \kappa_i^2  - C H\psi^{1-\alpha}\\
 \geq \,& \frac{\delta}{2} \sum F^{ii} + \delta_0 f_i \kappa_i^2  - C H\psi^{1-\alpha},
\end{aligned}
\end{equation}
where $\delta_0 := \frac{a}{2 (\max_{\ol \Omega} v - a)} > 0$.

We consider two cases. The positive constant $\epsilon_0$ will be determined later.

\bigskip

\noindent
{\bf Case 1.} \ $|h_{ii}| \leq \epsilon_0 h_{11}$ for all $i\geq2$.

\bigskip

In this case, we have
\[
[1-(n-2) \epsilon_0]h_{11} \leq \eta_{22}
\leq \cdots \leq \eta_{nn} \leq [1+(n-2) \epsilon_0]h_{11}.
\]
It then follows that
\[
\sigma_{n-1}(\eta) \geq \eta_{22} \cdots \eta_{nn}
\geq (1- (n-1) \epsilon_0)^{n-1}h_{11}^{n-1}.
\]
Choosing $\epsilon_0$ sufficiently small and using $n\geq 3$,
\begin{equation}\label{Case 1 eqn 1}
\sigma_{n-1}(\eta) \geq \frac{h_{11}^{n-1}}{2} \geq \frac{h_{11}^2}{2} \geq \frac{H^2}{2n^2}.
\end{equation}
We obtain, at $X_0$,
\begin{equation}\label{Case 1 eqn 2}
\sum_{i}F^{ii} = (n-1)\sum_{i}\hat{F}^{ii}
= (n-1) \sigma_{n-1}(\eta) \geq \delta_1 H^2
\end{equation}
for some positive constant $\delta_1 = \frac{n-1}{2n^2}$.
Thus, by \eqref{S-2} and \eqref{Case 1 eqn 2}, we have
\[
H \leq \frac{2C}{\delta \delta_1} (\sup \psi)^{1-\alpha}
\]
and \eqref{S-1} is proved.

\bigskip

\noindent
{\bf Case 2.} \ $h_{22}>\epsilon_0 h_{11}$ or $h_{nn}<-\epsilon_0 h_{11}$ .

\bigskip

By the definitions of $F^{ii}$ and $\hat{F}^{ii}$ and \eqref{NM-1},
\begin{equation}\label{Case 2 eqn 1}
F^{22} = \sum_{i\neq 2}\hat{F}^{ii} \geq \hat{F}^{11} \geq \frac{1}{n}\sum_{i}\hat{F}^{ii} =
 \frac{1}{n}\sigma_{n-1}(\eta) \geq \frac{c_0}{n} \sigma_1^{\alpha} (\eta) \sigma_n^{1-\alpha} (\eta)
     = \delta_2 H^\alpha \psi^{1-\alpha},
\end{equation}
where $\delta_2 = \frac{c_0(n-1)^\alpha}{n}$. Similarly, we have
\[
F^{nn} \geq \delta_2 H^\alpha \psi^{1-\alpha}.
\]
In this case, we have
\[
f_i \kappa_i^2 = F^{ii}h_{ii}^{2}
\geq F^{22}h_{22}^{2}+F^{nn}h_{nn}^{2}
\geq \epsilon_0^{2} \delta_2 h_{11}^{2} \geq \frac{\epsilon_0^{2} \delta_2}{n^2} H^{2+\alpha} \psi^{1-\alpha}.
\]
Thus, by \eqref{S-2}, we obtain
\[
H \leq \left(\frac{Cn^2}{\delta_0 \delta_2 \epsilon_0^2}\right)^{1/(1+\alpha)}
\]
and Theorem \ref{Thm-second-interior} follows immediately.
\end{proof}

\section{Boundary estimates for second order derivatives}
In this section, we establish the boundary estimates of second order derivatives.
\begin{theorem}
\label{thm-boundary}
Suppose $\Omega$ is a bounded domain in $\mathbb{R}^n$ with smooth strictly convex boundary $\partial \Omega$.
Assume that \eqref{add-8} and \eqref{add-1} hold.
Let $u \in C^3 (\overline{\Omega})$ be an admissible solution of \eqref{js-2}. Then there
exists a positive constant $C$ depending only on $\|u\|_{C^1 (\overline{\Omega})}$,
$\|\psi^{1/(n-1)}\|_{C^{1} (\overline{\Omega} \times [-\mu_0, \mu_0] \times \mathbb{S}^n)}$ and $\partial \Omega$
satisfying
\begin{equation}
\label{B2-0}
\max_{\partial \Omega} |D^2 u| \leq C,
\end{equation}
where $\mu_0 := \|u\|_{C^0 (\overline{\Omega})}$.
\end{theorem}

To prove \eqref{B2-0}, we consider
an arbitrary point $x_0 \in \partial \Omega$. Without loss of generality, we may assume that $x_0$ is the origin and that the
positive $x_n$-axis is in the interior normal direction to $\partial \Omega$ at the origin.
Suppose near the origin, the boundary $\partial \Omega$ is given by
\begin{equation}
\label{BC2-1}
x_n = \rho (x') = \frac{1}{2} \sum_{\alpha < n} \kappa^b_\alpha x_\alpha^2 + O (|x'|^3),
\end{equation}
where $\kappa^b_1, \ldots, \kappa^b_{n-1}$ are the principal curvatures of $\partial \Omega$ at the origin
and $x' = (x_1, \ldots, x_{n-1})$. Differentiating the
boundary condition $u = 0$ on $\partial \Omega$ twice, we can find a constant $C$ depending on
$\|u\|_{C^1 (\ol \Omega)}$ satisfying
\begin{equation}
\label{BC2-2}
|u_{\alpha \beta} (0)| \leq C \mbox{  for } 1\leq \alpha, \beta \leq n - 1.
\end{equation}
Since $\nu = \frac{(-Du, 1)}{w}$, $\psi$ can be regarded as a function of $x$, $u$ and $Du$. In the following, we denote
\[
\psi (x, u, D u) = \psi (X, \nu(X)) = \psi \left(x, u, \frac{(-Du, 1)}{w}\right)
\]
and the equation \eqref{js-2} can be written as
\begin{equation}
\label{1-1-1}
G (D^2 u, Du) := f (\lambda (A[u])) = \psi (x, u, Du),
\end{equation}
where $G = G (r, p)$ is viewed as a function of $(r, p)$ for $r \in S^{n \times n}$ and $p \in \mathbb{R}^n$.
Define
\begin{equation}
\label{BC2-25}
G^{ij} = \frac{\partial G}{\partial r_{ij}} (D^2 u, D u),\ \ G^{i} = \frac{\partial G}{\partial p_i} (D^2 u, D u),
  \ \ \psi_{u_i} = \frac{\partial \psi}{\partial u_i} (x, u, Du)
\end{equation}
and the linearized operator by
\[
L = G^{ij} \partial_{ij} - \psi_{u_i} \partial_i.
\]
The following lemma was proved in \cite{GS04}.
\begin{lemma} We have
\label{lemGS2}
\begin{equation}
\label{GS-2}
G^s = - \frac{u_s}{w} \sum_i f_i \kappa_i - \frac{2}{w (1+w)} \sum_{t,j}F^{ij} a_{it} \big(w u_t \gamma^{sj}
   + u_j \gamma^{ts}\big),
\end{equation}
where $a_{ij} = \frac{1}{w}\gamma^{ik}u_{kl}\gamma^{lj}$, $\kappa = \lambda (\{a_{ij}\})$, $f_i = \frac{\partial f (\kappa)}{\kappa_i}$ and
\[
F^{ij} = \frac{\partial f (\lambda (A[u]))}{\partial a_{ij}}.
\]
\end{lemma}

Next, we establish the estimate
\begin{equation}
\label{BC2-3}
|u_{\alpha n} (0)| \leq C \mbox{  for } 1\leq \alpha \leq n - 1.
\end{equation}
Let $\omega_\delta = \{x \in \Omega: \rho (x') < x_n < \rho (x') + \delta^2 , |x'| < \delta\}$. Since  $\Omega$ is
uniformly $2$-convex for $n \geq 3$ and strictly convex for $n=2$, 
there exist two positive constants $\theta$ and $K$ satisfying
\begin{equation}
\label{BC2-5}
\begin{aligned}
(\kappa_1^b - 3 \theta, \ldots, \kappa_{n-1}^b - 3 \theta, 2 K) \in \Gamma_{3} & \mbox{ for } n \geq 3\\
(\kappa_1^b - 3 \theta, \ldots, \kappa_{n-1}^b - 3 \theta, 2 K) \in \Gamma_{2} & \mbox{ for } n = 2.
\end{aligned}
\end{equation}
Define
\begin{equation}
\label{BC2-6}
v = \rho (x') - x_n - \theta |x'|^2 + K x_n^2.
\end{equation}
We find that the boundary $\partial \omega_\delta$ consists of three parts: $\partial \omega_\delta
= \partial_1 \omega_\delta \cup \partial_2 \omega_\delta \cup \partial_3 \omega_\delta$, where
$\partial_1 \omega_\delta$, $\partial_2 \omega_\delta$ are  defined by $\{x_n=\rho\} \cap\overline{\omega}_{\delta}$, $\{ x_n=\rho+\delta^2\}\cap\overline{\omega}_{\delta}$
respectively, and $\partial_3 \omega_\delta$ is defined by $\{|x'| = \delta\}\cap\overline{\omega}_{\delta}$.

We see that when $\delta$ depending on $\theta$ and $K$ is sufficiently small, we have
\begin{equation}
\label{BC2-12}
\begin{aligned}
v \leq & - \frac{\theta}{2} |x'|^2, & \mbox{ on } \partial_1 \omega_\delta\\
v \leq & - \frac{\delta^2}{2}, & \mbox{ on } \partial_2 \omega_\delta\\
v \leq & - \frac{\theta \delta^2}{2},   & \mbox{ on } \partial_3 \omega_\delta.
\end{aligned}
\end{equation}
In view of \eqref{BC2-1} and \eqref{BC2-5},  $v$ is $3$-convex on $\overline{\omega}_\delta$ for $n\geq 3$
and strictly convex for $n=2$.
Thus, there exists an uniform constant $\eta_0 > 0$ depending only on $\theta$, $\p\Omega$
and $K$ satisfying
\[
\lambda (D^2 v - 2 \eta_0 I) \in \Gamma_{3} \mbox{ for } n \geq 3
  \mbox{ and } \lambda (D^2 v - 2 \eta_0 I) \in \Gamma_{2} \mbox{ for } n=2 \mbox{ on } \overline{\omega}_\delta.
\]
By Proposition \ref{prop-I}, we have
\begin{equation}
\label{BC2-4}
\lambda \left(\frac{1}{w} \{\gamma^{is} (v_{st} - 2 \eta_0 \delta_{st}) \gamma^{jt}\}\right) \in \Gamma_2 \subset \Gamma  \mbox{ on } \overline{\omega}_\delta.
\end{equation}
As \cite{Ivochkina90}, we consider
\[
W := \nabla'_\alpha u - \frac{1}{2} \sum_{1\leq \beta \leq n - 1} u_\beta^2
\]
defined on $\ol \Omega_\delta$ for some small $\delta$,
where
\[
\nabla'_\alpha u := u_\alpha + \rho_\alpha u_n, \mbox{ for } 1\leq \alpha \leq n - 1.
\]
Similar to Lemma 3.3 of \cite{JW21}, we need the following lemma.
\begin{lemma} If $\delta$ is sufficiently small, we have
\label{BC2-lem1}
\begin{equation}
\label{BC2-15}
LW \leq C \left(\psi^{1-1/(n-1)} + \psi |D W| + \sum_i G^{ii} + G^{ij} W_i W_j\right),
\end{equation}
where $C$ is a positive constant depending on $n$, $\|u\|_{C^1 (\ol \Omega)}$, $\|\psi^{1/(n-1)}\|_{C^1 (\ol \Omega \times [-\mu_0, \mu_0]\times \mathbb{S}^n)}$ and $\partial \Omega$, where $\mu_0 = \|u\|_{C^0 (\overline{\Omega})}$.
\end{lemma}
\begin{proof}
We first note that, by differentiating the equation \eqref{1-1-1},
\begin{equation}
\label{BC2-16-1}
\begin{aligned}
G^{ij} W_{ij} + G^s W_s
  \leq \,& \nabla'_\alpha \psi - \sum_{\beta\leq n-1} u_\beta \psi_\beta + 2 G^{ij} u_{ni} \rho_{\alpha j} \\
     & - \sum_{\beta \leq n-1} G^{ij} u_{\beta i} u_{\beta j}
      + u_n G^{ij} \rho_{\alpha ij} + u_n G^s \rho_{\alpha s}.
\end{aligned}
\end{equation}
By \eqref{add-1} and \eqref{BC2-16-1}, we obtain
\begin{equation}
\label{BC2-16}
\begin{aligned}
L W + G^s W_s
  \leq \,& C \psi^{1-1/(n-1)} + 2 G^{ij} u_{ni} \rho_{\alpha j} \\
     & - \sum_{\beta \leq n-1} G^{ij} u_{\beta i} u_{\beta j}
      + u_n G^{ij} \rho_{\alpha ij} + u_n G^s \rho_{\alpha s}.
\end{aligned}
\end{equation}
We have
\[
G^{ij} = \frac{1}{w}\sum_{s,t} F^{st} \gamma^{is} \gamma^{tj} \mbox{ and }
u_{ij} = w\sum_{s,t} \gamma_{is} a_{st} \gamma_{tj}.
\]
It follows that
\[
\sum_{\beta \leq n - 1} G^{ij} u_{\beta i} u_{\beta j}
  = w \sum_{\beta \leq n - 1} \sum_{s,t}F^{ij} \gamma_{\beta s} \gamma_{\beta t} a_{si} a_{tj}.
\]
By \cite{GS04}, we can find an orthogonal matrix $B = (b_{ij})$ that can diagonalize $(a_{ij})$ and $(F^{ij})$ at the same time:
\[
F^{ij} =\sum_s b_{is} f_s b_{js} \mbox{ and } a_{ij} =\sum_s b_{is} \kappa_s b_{js}.
\]
Therefore, we get
\[
\sum_{\beta \leq n - 1} G^{ij} u_{\beta i} u_{\beta j}
  = w \sum_{\beta \leq n - 1} \sum_i\left(\sum_s\gamma_{\beta s} b_{si}\right)^2  f_i \kappa_i^2.
\]
Let the matrix $\eta = (\eta_{ij}) = (\sum_s\gamma_{is} b_{sj})$. We find $\eta \cdot \eta^T = g$ and $|\det (\eta)| = \sqrt{1 + |Du|^2}$.
Therefore, we obtain
\begin{equation}
\label{BC2-14}
\sum_{\beta \leq n - 1} G^{ij} u_{\beta i} u_{\beta j}
  = w \sum_{\beta \leq n - 1} \sum_i\eta_{\beta i}^2  f_i \kappa_i^2.
\end{equation}
We have
\begin{equation}
\label{BC2-19}
G^{ij} u_{ni} \rho_{\alpha j} = \sum_{i,t}f_i \kappa_i b_{si} \gamma^{js} b_{ti} \gamma_{nt} \rho_{\alpha j} \leq C \sum_i f_i |\kappa_i|.
\end{equation}
For any indices  $j, t$, we have
\[
F^{ij} a_{it} =\sum_{i,s,p} b_{is} f_s b_{js} b_{ip} \kappa_p b_{tp} = \sum_i f_i \kappa_i b_{ji} b_{ti}.
\]
Thus,
by \eqref{GS-2}, we find
\begin{equation}
\label{BC2-11}
\mid G^s \rho_{\alpha s} \mid \leq C \sum_i f_i |\kappa_i|.
\end{equation}
Combining \eqref{BC2-16}-\eqref{BC2-11},
we obtain
\begin{equation}
\label{BC2-26}
\begin{aligned}
LW + G^s W_s
     \leq \,& C \left(\psi^{1-1/(n-1)} + \sum_i G^{ii} + \sum_i f_i |\kappa_i|\right)\\
        & - w \sum_{\beta \leq n - 1} \sum_i\eta_{\beta i}^2 f_i \kappa_i^2.
\end{aligned}
\end{equation}
Now we consider the term $G^s W_s$. We have, by \eqref{GS-2} and the definition of the matrix $(b_{ij})$,
\begin{equation}
\label{BC2-17}
\begin{aligned}
- G^s W_s 
   = \,& \frac{1}{w} \sum_s\left(n \psi u_s + 2 \sum_{t,i}f_i \kappa_i (b_{ti} u_t) \gamma^{sl} b_{li}\right) W_s\\
   \leq \,& C \psi |D W| + \frac{2}{w} \sum_{t,i}f_i \kappa_i (b_{ti} u_t) \gamma^{sl} b_{li} W_s.
\end{aligned}
\end{equation}
We divide into two cases to further discuss: (a) $\sum_{\beta \leq n - 1} \eta_{\beta i}^2 \geq \epsilon$ for all $i$; and
(b) $\sum_{\beta \leq n - 1} \eta_{\beta r}^2 < \epsilon$ for some index $1 \leq r \leq n$, where
 $\epsilon$ is some positive constant,  which will be determined later.

For the case (a), by \eqref{BC2-14}, we have
\[
\sum_{\beta \leq n - 1} G^{ij} u_{\beta i} u_{\beta j}
  \geq \epsilon \sum_{i} f_i \kappa_i^2.
\]
By Cauchy-Schwarz inequality, we get
\begin{equation}
\label{BC2-31}
\frac{2}{w} \kappa_i (b_{ti} u_t) \gamma^{sl} b_{li} W_s
  \leq \frac{\epsilon}{2} \kappa_i^2 + \frac{C}{\epsilon } (\gamma^{sl} b_{li} W_s)^2.
\end{equation}
It follows that
\[
- G^s W_s \leq C \psi |D W| + \frac{\epsilon}{2} f_i \kappa_i^2 + \frac{C}{\epsilon} G^{ij} W_i W_j.
\]
It is clear that, for any $\epsilon_1 > 0$,
$$\sum_i f_i |\kappa_i|\leq \frac{1}{\epsilon_1} \sum_i f_i + \epsilon_1 \sum_i f_i \kappa_i^2
   \leq C\left(\frac{1}{\epsilon_1} \sum_i G^{ii}+ \epsilon_1 \sum_i f_i \kappa_i^2\right).$$
Combining the previous four inequalities with \eqref{BC2-26}, \eqref{BC2-15} follows.

For the case (b), as in \cite{Ivochkina90}, we have
\[
1 \leq \det (\eta) \leq \eta_{nr} \det (\eta') + C_1 \epsilon
  \leq \sqrt{1 + \mu_1^2} \mid\det (\eta')\mid + C_1 \epsilon,
\]
where $\eta' := \{\eta_{\alpha \beta}\}_{\alpha\neq n, \beta \neq r}$, $\mu_1 := \|D u\|_{C^0 (\overline{\Omega})}$ and $C_1$ is a positive constant depending
only on $n$ and $\mu_1$.
We choose  sufficiently small $\epsilon$ satisfying $C_1\epsilon < \frac{1}{2}$. Thus, we get
\[
\mid\det (\eta')\mid \geq \frac{1}{2 \sqrt{1 + \mu_1^2}}.
\]
On the other hand, for any fixed $\alpha\neq r $, we have
\[
\mid\det (\eta')\mid \leq C \sum_{\beta \neq n} |\eta_{\beta \alpha}|.
\]
for some positive constant $C$ depending only on $n$ and $\mu_1$. Therefore, combining the above two inequalities, we get, for any $i\neq r$,
$$\sum_{\beta\leq n-1}\eta^2_{\beta i}\geq c_1$$ for some positive constant $c_1$ depending on $\|u\|_{C^1 (\overline{\Omega})}$, which implies, in view of \eqref{BC2-14},
\begin{equation}
\label{BC2-13}
\sum_{\beta \leq n - 1} G^{ij} u_{\beta i} u_{\beta j} \geq c_1 \sum_{i \neq r} f_i  \kappa_i^2.
\end{equation}
If $\kappa_r \leq 0$, by Lemma 2.20 of \cite{Guan14}, we have
\[
 \sum_{i \neq r} f_i  \kappa_i^2
    \geq \frac{1}{n+1} \sum_{i = 1}^n f_i \kappa_i^2.
\]
Thus \eqref{BC2-15} follows using a similar argument as the Case (a).

Hence, in the following, we may assume $\kappa_r > 0$. Without loss of generality, assume $r = 1$.
Now we consider two cases.

\noindent
{\bf Case (b-1).} \ $|\kappa_i| \leq \epsilon_0 \kappa_1$ for all $i\geq2$, where the positive constant
$\epsilon_0$ will be chosen later.

In this case, as in Section 4, we find,
\[
(1 - (n-2)\epsilon_0) \kappa_1 \leq \lambda_i \leq (1 + (n+2)\epsilon_0) \kappa_1, \mbox{ for } i \geq 2,
\]
where $\lambda_1, \ldots, \lambda_n$ are defined in \eqref{lambda}.
By the equation \eqref{js-2}, fixing the constant $\epsilon_0$ sufficiently small,
\[
\lambda_1 = \frac{\psi}{\lambda_2 \cdots \lambda_n} \leq C \kappa_1^{1-n} \psi.
\]
It follows that
\[
\sigma_{n-1;i} (\lambda) = \prod_{j\neq i} \lambda_j \leq C \kappa_1^{n-2} \kappa_1^{1-n} \psi = C \kappa_1^{-1} \psi,
   \mbox{ for } i \geq 2.
\]
Therefore,
\[
f_1 (\kappa) = \sum_{i \neq 1} \sigma_{n-1; i} (\lambda) \leq C \kappa_1^{-1} \psi
\]
and
\[
f_1 \kappa_1 \leq C \psi.
\]
By \eqref{BC2-17} and \eqref{BC2-31} and the Cauchy-Schwarz inequality, we have, for any $\epsilon > 0$,
\begin{equation}
\label{BC2-30}
\begin{aligned}
- G^s W_s \leq \,& C \psi |DW| + \frac{2}{w} \sum_{i \neq 1}f_i \kappa_i (b_{ti} u_t) \gamma^{sl} b_{li} W_s\\
  \leq \,& C \psi |DW| + C \sum_{i\neq 1} f_i |\kappa_i| |\gamma^{sl} b_{li} W_s| \\
  \leq \,& C \psi |DW| + \epsilon \sum_{i \neq 1} f_i \kappa_i^2
     + \frac{C}{\epsilon} \sum_{i=1}^n f_i \gamma^{sl} b_{li} W_s \gamma^{tk} b_{ki} W_t\\
  \leq \,& C \psi |DW| + \epsilon \sum_{i \neq 1} f_i \kappa_i^2 + \frac{C}{\epsilon} G^{ij} W_i W_j.
\end{aligned}
\end{equation}
Using \eqref{BC2-13}, \eqref{BC2-15} is proved by fixing $\epsilon$ sufficiently small.

\noindent
{\bf Case (b-2).} \ $|\kappa_{i_0}| > \epsilon_0 \kappa_1$ for some $i_0 \geq2$,

First we find
\[
\begin{aligned}
\gamma^{sl} b_{l1} W_s = \,& \gamma^{sl} b_{l1} \Big(u_{\alpha s} + \rho_\alpha u_{ns} - \sum_{\beta \leq n-1} u_\beta u_{\beta s}
   + \rho_{\alpha s} u_n\Big)\\
     = \,& w \Big(\eta_{\alpha 1} + \rho_\alpha \eta_{n 1} - \sum_{\beta \leq n-1} u_\beta \eta_{\beta 1}
   \Big) \kappa_1 + \gamma^{sl} b_{l1} \rho_{\alpha s} u_n.
\end{aligned}
\]
It follows that
\begin{equation}
\label{BC2-22}
\mid\gamma^{sl} b_{l1} W_s\mid \leq C w (\epsilon + |\rho_\alpha|) \kappa_1 + C.
\end{equation}
It is obvious that
\[
f_1 \kappa_1 = n \psi - \sum_{i \neq 1} f_i \kappa_i.
\]
Then we have
\[
\begin{aligned}
 \frac{2}{w} f_1 \kappa_1 \mid(\sum_tb_{t1} u_t) \,& \gamma^{sl} b_{l1} W_s\mid = \frac{2}{w} \big(n\psi - \sum_{i \neq 1} f_i \kappa_i\big)
    \mid(\sum_tb_{t1} u_t) \gamma^{sl} b_{l1} W_s\mid  \\
 \leq \,& C \psi |D W| + C (\epsilon + |\rho_\alpha|) \sum_{i \neq 1} f_i |\kappa_i| \kappa_1 + C \psi + C \sum_{i\neq 1} f_i |\kappa_i|\\
\leq \,& C \psi |D W| + \left(\epsilon_0^{-1} C (\epsilon + |\rho_\alpha|) + \epsilon_1\right) \sum_{i\neq 1} f_i \kappa_i^2
   + C \psi + \frac{C}{\epsilon_1} \sum_{i\neq 1} f_i
\end{aligned}
\]
for any $\epsilon_1 > 0$.
We can choose sufficiently small $\delta$ $\epsilon$ and $\epsilon_1$ satisfying
\[
\left(\epsilon_0^{-1} C (\epsilon + |\rho_\alpha|) + \epsilon_1\right) < \frac{c_1}{4}.
\]
Therefore, \eqref{BC2-15} follows by \eqref{BC2-13} as in Case (b-1).
\end{proof}
To proceed we consider the following function on $\overline{\Omega}_\delta$, for sufficiently small $\delta$,
\begin{equation}
\label{BC2-20}
\Psi := v - td + \frac{N}{2} d^2,
\end{equation}
where $v(x)$ is defined by \eqref{BC2-6}, $d (x) := \mathrm{dist} (x, \partial \Omega)$ is the distance from $x$ to the boundary $\partial \Omega$,
$t,N$ are two positive constants to be determined later. We first suppose $\delta < 2t/N$ so that
\begin{equation}\label{new3.1}
- t d + \frac{N}{2} d^2 \leq 0 \mbox{ on } \overline{\Omega}_\delta.
\end{equation}
Let
\begin{equation}\label{new3.2}
\tilde{W} := 1 - \exp\{- b W\}.
\end{equation}
By \eqref{BC2-15}, we can choose the constant $b$ sufficiently large such that
\begin{equation}\label{new3.4}
L \tilde{W} \leq C (\psi^{1-1/(n-1)} + \psi |D \tilde{W}| + \sum_i G^{ii}).
\end{equation}
We consider the function
\[
\Phi := R \Psi - \tilde{W},
\]
where $R$ is a positive constant sufficiently large to be chosen. We shall prove
\begin{equation}
\label{add-6}
\Phi \leq 0 \mbox{ on } \overline{\Omega}_\delta
\end{equation}
by choosing suitable constants $\delta$, $t$, $N$ and $R$. We first deal with the case that the maximum of $\Phi$
is attained at an interior point $x_0 \in \Omega_\delta$. Now we consider two cases: (i) $\psi (x_0, u(x_0), Du(x_0)) \geq \epsilon_0$
and (ii)  $\psi (x_0, u(x_0), Du(x_0)) < \epsilon_0$, where $\epsilon_0$ is a positive constant to be determined later.

\textbf{Case (i).}
Since $f^{1/n}$ is concave in $\Gamma$ and homogeneous of degree one and $|Dd| \equiv 1$ on the boundary $\partial \Omega$, we have, by \eqref{BC2-4} and \eqref{cj-2},
\[
\begin{aligned}
& \frac{1}{n} \psi^{\frac{1}{n} - 1}G^{ij} (D^2 v - \eta_0 I + N D d \otimes D d)_{ij}\\
   \geq \,& G^{1/n} (D^2 v - \eta_0 I + N D d \otimes D d, Du)
   \geq \mu (N),
\end{aligned}
\]
at $x_0$, where $\lim_{N \rightarrow + \infty}\mu (N) = +\infty$.
Consequently, at $x_0$ we have
\[
\begin{aligned}
G^{ij} \Psi_{ij} \geq \,& n \epsilon_0^{1-1/n}\mu(N) + \eta_0 \sum_i G^{ii} + (N d - t) G^{ij} d_{ij}\\
  \geq \,& 2 \mu_1(N) + (\eta_0 - C N \delta - C t) \sum_i G^{ii},
\end{aligned}
\]
where $\mu_1 (N) := n \epsilon_0^{1-1/n}\mu(N)/2$.
Note that
\[
|D \Psi| = |D v - t D d + N d D d| \leq C (1+t) + C\delta N \leq \mu_1(N)^{1/2} \mbox{ in } \Omega_\delta,
\]
if $N$ is sufficiently large and $\delta < \sqrt{\mu_1 (N)}/2CN$.
Therefore, we further take $\delta$ and $t$ sufficiently small such that $C N \delta + C t < \eta_0/2$. We thus obtain, at $x_0$,
\begin{equation}
\label{BC2-21}
G^{ij} \Psi_{ij} \geq \mu_1 (N) + \mu_1(N)^{1/2} |D \Psi| + \frac{\eta_0}{2} \sum_i G^{ii}.
\end{equation}
Therefore, by \eqref{new3.4} and \eqref{BC2-21} we have, at $x_0$, 
\[
\begin{aligned}
0 \geq L \Phi
  \geq \,& R\mu_1 (N) + R\mu_1(N)^{1/2} |D \Psi| - R\psi_{u_i} \Psi_i + \frac{R\eta_0}{2} \sum G^{ii}\\
     & - C \left(1 + |D \tilde{W}| + \sum G^{ii}\right)\\
 \geq \,& R \mu_1(N) - C + R (\mu_1(N)^{1/2} - C)|D \Psi| + \left(\frac{R\eta_0}{2} - C\right)\sum G^{ii} > 0
\end{aligned}
\]
provided $N$ and $R$ are sufficiently large which is a contradiction.

\textbf{Case (ii)} Let $H$ denote the mean curvature of $M_u$ at $X_0 = (x_0, u(x_0))$. We consider two cases:
(ii-a) $H \leq A$ and (ii-b) $H > A$, where $A$ is a positive constant sufficiently large to be chosen.

\textbf{Case (ii-a).} Let
\[
\tilde{a}_{ij} = w^{-1} g^{il} u_{lj}.
\]
First we note that  $|\tilde{a}_{ij} (x_0)| \leq \hat{C} H \leq \hat{C}A$, $i,j = 1, \ldots, n$, for some positive constant
$\hat{C}$ depending only on $n$ since $g^{-1} \eta$ is positive definite.
We see, at $x_0$,
\[
\begin{aligned}
0 = \,& \Phi_i = (R \Psi - \tilde{W})_i\\
  = \,& R (v_i - t d_i + Ndd_i) - b \exp\{-bW\} \Big(u_{\alpha i}
  + \rho_\alpha u_{ni} + \rho_{\alpha i} u_n - \sum_{\beta \leq n-1} u_\beta u_{\beta i}\Big).
\end{aligned}
\]
for $i = 1, \ldots, n$,
Therefore, we get
\[
\begin{aligned}
\ & \tilde{a}_{n\alpha} + \rho_\alpha \tilde{a}_{nn} +w^{-1}g^{ni}\rho_{i\alpha} u_n- \sum_{\beta \leq n-1} u_\beta \tilde{a}_{n \beta}\\
=\ & w^{-1} g^{ni} \Big(u_{i \alpha}
  + \rho_\alpha u_{in} + \rho_{i \alpha} u_n - \sum_{\beta \leq n-1} u_\beta u_{i \beta}\Big)\\
 = \ & R b^{-1} \exp\{b W\}w^{-1}\big(g^{ni} v_i - tg^{ni} d_i + Nd g^{ni}d_i\big).
\end{aligned}
\]
Note that $v_n (0) = -1$ and that $v_\gamma (0) = d_\gamma (0) = 0$ for $1\leq \gamma \leq n-1$, we see
at $x_0$, if we let $\delta$ and $t$ be sufficiently small,
\[
\tilde{a}_{n\alpha} \leq - \hat{c} R + C A,
\]
where $\hat{c}$ and $C$ are positive constants depending only on $\|u\|_{C^1 (\overline{\Omega})}$.
We then get a contradiction
if $R \gg A$ is sufficiently large since
$|\tilde{a}_{n\alpha}(x_0)| \leq \hat{C}A$.

\textbf{Case (ii-b).} Note that, by \eqref{NM-1},
\[
\sum G^{ii} = \frac{1}{w} g^{ij} F^{ij} \geq \gamma_1 \sigma_{n-1} (\eta) \geq \gamma_2 H^{1/(n-1)} \psi^{1-1/(n-1)}
   \geq \gamma_2 A^{1/(n-1)} \psi^{1-1/(n-1)}
\]
for some positive constant $\gamma_2$ depending only on $n$ and $\|u\|_{C^1 (\overline{\Omega})}$.
By \eqref{add-1}, we find
\[
|\psi_{u_i}| \leq C \psi^{1-1/(n-1)}, \mbox{ for all }i = 1, \ldots, n.
\]
Similar to \eqref{BC2-21}, we have
\[
G^{ij} \Psi_{ij} \geq \frac{\eta_0}{2} \sum G^{ii}
\]
and
\begin{equation}
\label{add-7}
\begin{aligned}
L \Psi \geq \,& \frac{\eta_0}{2} \sum G^{ii} - C \psi^{1-1/(n-1)}\\ \geq \,&
  \frac{\eta_0}{4} \sum G^{ii} + \left(\frac{\eta_0 \gamma_2}{4} A^{1/(n-1)} - C\right)\psi^{1-1/(n-1)}\\
  \geq \,& \frac{\eta_0}{4} \sum G^{ii} + \frac{\eta_0 \gamma_2}{8} A^{1/(n-1)}\psi^{1-1/(n-1)}
\end{aligned}
\end{equation}
by fixing $A$ sufficiently large. Next, we see
\[
D \tilde{W} (x_0) = R D \Psi (x_0).\]
Combining \eqref{new3.4} and \eqref{add-7}, we have, at $x_0$,
\[
\begin{aligned}
0 \geq \,& L \Phi \\
\geq \,& \left(\frac{\theta R}{16} - C\right)\sum G^{ii} + \left(\frac{\theta \gamma_2 R}{32} A^{1/(n-1)} - C\right)
  \psi^{1-1/(n-1)} - C \psi > 0
\end{aligned}
\]
if we fix $R$ sufficiently large and $\epsilon_0$ sufficiently small, which is a contradiction.

In view of Case (i) and Case (ii), the function $\Phi$ cannot attain its maximum at an interior point of $\Omega_\delta$
when $R$, $N$ are large enough and $\delta$, $t$ are small enough. By \eqref{BC2-12} and \eqref{new3.1}, we can fix
$\delta$ smaller and $R$ larger such that $\Phi \leq 0$ on $\partial \Omega_\delta$. Thus, \eqref{add-6} is proved.
Consequently, we have $(R \Psi - \tilde{W})_n (0) \leq 0$ since  $(R \Psi - \tilde{W}) (0) = 0$. Therefore, we obtain
\[
u_{n \alpha} (0) \geq - C,
\]
The above arguments also hold with respect to $- \nabla'_\alpha u - 1/2 \sum_{\beta \leq n - 1} u_\beta^2$.
Hence, we obtain \eqref{BC2-3}.

It suffices to establish the upper bound of $u_{nn}(0)$ since $H [M_u] > 0$.
By \eqref{add-8} and \eqref{NM-2}, we have
\begin{equation}
\label{BC2-9}
H [M_u] = \sigma_1 (\kappa [M_u]) = \frac{1}{n-1} \sigma_1 (\eta) \geq c_0 \sigma_n^{1/n} (\eta) \geq c_0 \ul \psi^{1/n}
\end{equation}
for some $c_0$ depending only on $n$.
Since $\ul \psi \not \equiv 0$ on $[\inf u, 0]$, there exists a point $x_0 \in \ol \Omega$ such that
\[
\ul \psi (x_0, u (x_0)) \geq \delta_0 > 0
\]
for some positive constant $\delta_0$.
By the continuity of $\ul \psi$ and that we have establish the $C^1$ estimates,
there exists a neighbourhood $U$ of $x_0$ such that
\[
\inf_{U \cap \overline{\Omega}} \ul \psi (x, u(x)) \geq \frac{\delta_0}{2}.
\]
Fix $U_1 \subset \subset U$. As \cite{JW21}, let $\overline{\psi}$ be a smooth function such that
$\overline{\psi} = c_0 (\delta_0/2)^{1/n}$ in $U_1$,
$\overline{\psi} \leq c_0 (\delta_0/2)^{1/n}$ in $U - U_1$ and $\overline{\psi} = 0$ outside $U$.

Next we prove
\begin{equation}
\label{BC2-10}
u_n (0) \leq - \gamma_1
\end{equation}
for some unform positive constant $\gamma_1$.



By Theorem 16.10 of \cite{GT98}, there exists a unique solution $\ol u$ of the prescribed mean curvature equation
\[
\sigma_1 (\kappa[M_{\ol u}]) = \epsilon_2 \overline{\psi} \mbox{ in } \Omega
\]
with $\ol u = 0$ on $\partial \Omega$,  if the positive constant $\epsilon_2<1$ is sufficiently small.
Since the above equation is uniform elliptic, by Hopf's lemma,
we see $\ol u < 0$ in $\Omega$ and
$\ol u_\mu < 0$ on $\partial \Omega$, where $\mu$ is the unit interior normal with respect to $\partial \Omega$.
Since $\partial \Omega$ is compact, there exists a uniform constant $\gamma_1 > 0$ such that
$\ol u_\mu \leq - \gamma_1$ on $\partial \Omega$. By \eqref{BC2-9}, we have
$$\sigma_1 (\kappa[M_{u}]) \geq \sigma_1 (\kappa[M_{\ol u}]) \mbox{ in } \Omega.$$
Therefore, by the maximum principle and $u = \ol u = 0 \mbox{ on } \partial \Omega$, we find
\[
u \leq \ol u \mbox{ in } \Omega
\]
and $u_n (0) \leq \ol u_n (0) \leq - \gamma_1$ which is \eqref{BC2-10}.

It is clear that, at the origin,
\[
u_{\alpha \beta} = - (u_n) \kappa^b_\alpha \delta_{\alpha \beta}, \mbox{ for } 1 \leq \alpha, \beta \leq n-1
\]
and
\[
g^{ij} = \delta_{ij} - \frac{|Du|^2}{w^2} \delta_{in} \delta_{jn}.
\]
It follows that, at the origin,
\[
S_{1;n} (D^2 u, Du) = \sum_{\alpha \leq n-1} u_{\alpha \alpha} = - u_n \sigma_1 (\kappa^b) \geq \gamma_2 > 0
\]
for some uniform positive constant $\gamma_2$ by \eqref{BC2-10}.
Thus, by the equation \eqref{js-2}, \eqref{js-5}, \eqref{BC2-2} and \eqref{BC2-3},
we get upper bound of $u_{nn}(0)$. Theorem \ref{thm-boundary} is proved.


Now we prove Theorem \ref{js-thm2}. We note that we have also established the $C^2$ estimates for \eqref{js-2} when
it is non-degenerate.
The $C^{2, \alpha}$ estimates can be established by Evans-Krylov theory since the equation \eqref{js-2}
is concave uniformly elliptic with respect to admissible solutions in the non-degenerate case.
Higher order estimates can be derived using Schauder theory. Then Theorem \ref{js-thm2} can be proved
by standard arguments using the continuity method. The reader is referred to \cite{CNSV} for details.

Finally, Theorem \ref{js-thm1} can be proved by approximation as in \cite{JW21}.

\bigskip

{\bf Conflict of interest statement.} The authors declare that there is no conflict of interest.

\end{document}